\documentclass{article}
\usepackage{amsfonts}
\usepackage{amsmath}

\newtheorem{theorem}{Theorem}
\newtheorem{acknowledgement}[theorem]{Acknowledgement}

\newtheorem{definition}[theorem]{Definition}

\newtheorem{lemma}[theorem]{Lemma}

\newenvironment{proof}[1][Proof]{\noindent\textbf{#1.} }{\ \rule{0.5em}{0.5em}}
\input{tcilatex}
\begin{document}

\title{Existence and uniqueness of solution for the stochastic nonlinear
diffusion equation of plasma}
\author{Ioana Ciotir \\
%EndAName
Departament of Mathematics, \\
Faculty of Economics and Business Administration \\
"Al.I. Cuza" University, \\
Bd. Carol no. 9-11, Iasi, Romania}
\maketitle

\begin{abstract}
In this paper we are concerned with the stochastic partial differential
equations of superfast diffusion processes describing behavior of plasma $%
dX\left( t\right) -\Delta ~sign\left( X\left( t\right) \right) \ln \left( | X\left( t\right)|
+1\right)  dt=\sqrt{Q}%
dW\left( t\right) ,$ in $\left( 0,T\right) \times \mathcal{O}$, where $%
\mathcal{O}$ is a bounded open interval of $\mathbb{R}$.
We define a strong solution adequate to the properties of the natural logarithm
and we prove the corresponding existence and uniqueness result.
\end{abstract}


MSC[2000] 76S05, 60H15, 82D10

\section*{Introduction}

Consider a nonlinear diffusion process of the following form%
\begin{equation}
dX\left( t\right) =\Delta \ln \left( X\left( t\right) +1\right) dt
\label{proces}
\end{equation}%
where $X\left( t,\xi \right) $ is the density for the time - space
coordinates $\left( t,\xi \right) .$ This equation describes the process
that has been observed during experiments using Wisconsin toroidal octupole
plasma containment device (see \cite{ehrhardt}). Kamimura and Dawson
predicted in \cite{Kamimura}\ this process for cross-field conservative
diffusion of plasma including mirror effects.

The same equation describes the expansion of a thermalized electron cloud
and arises also in studies of the central limit approximation to Carleman's
model of the Boltzmann equation (see \cite{carleman} and \cite{Kurtz}). The
asymptotic behavior of this equation was studied in \cite{berryman}. Most of
the natural phenomena exhibit variability which cannot be modeled by using
deterministic approaches. More accurately, natural systems can be
represented as stochastic models and the deterministic description can be
considered as the subset of the pertinent stochastic models.

The purpose of this paper is to analyze such equations within the framework
of stochastic evolution equations with (\ref{proces}) as underlying
motivating example. Let us now introduce the suitable framework for this
problem.

\textbf{Notation}

Let $\mathcal{O}$ be a bounded open interval of $\mathbb{R}$. Recall the
distribution spaces $H_{0}^{1}\left( \mathcal{O}\right) $ on $\mathcal{O}$\
and it's dual $H^{-1}\left( \mathcal{O}\right) $ with the scalar product and
the norm given by
\begin{equation*}
\left\langle u,v\right\rangle _{-1}=\left( \left( -\Delta \right)
^{-1}u,v\right) _{2},\quad \left\vert u\right\vert _{-1}=\left( \left(
-\Delta \right) ^{-1}u,u\right) _{2}^{1/2},
\end{equation*}%
respectively, where $\left( \cdot ,\cdot \right) _{2}$ is the pairing
between $H_{0}^{1}\left( \mathcal{O}\right) $ and $H^{-1}\left( \mathcal{O}%
\right) $ and the scalar product of $L^{2}\left( \mathcal{O}\right) .$ The
norm in $L^{p}\left( \mathcal{O}\right) $, $1\leq p\leq \infty $ will be
denoted by $\left\vert \cdot \right\vert _{p}.$

Given a Hilbert space $U$, the norm of $U$ will be denoted by $\left\vert
\cdot \right\vert _{U}$ and the scalar product by $\left( \cdot ,\cdot
\right) _{U}.$ By $C\left( \left[ 0,T\right] ;U\right) $ we shall denote the
space of $U-$ valued continuous functions on $\left[ 0,T\right] $ and by $%
C_{W}\left( \left[ 0,T\right] ;L^{2}\left( \Omega ,\mathcal{F},\mathbb{P}%
;U\right) \right) $ the space of all $U$-valued adapted stochastic processes
with respect to filtration $\mathcal{F}$ of the probability space, which are
mean square continuous.

\textbf{Formulation of the problem and hypotheses}

The main result is an existence and uniqueness theorem for the following
stochastic nonlinear diffusion equations in $H^{-1}\left( \mathcal{O}\right)
$ with additive noise%
\begin{equation}
\left\{
\begin{array}{ll}
dX\left( t\right) -\Delta ~sign\left( X\left( t\right) \right) \ln \left( |
X\left( t\right)| +1\right) dt=\sqrt{Q}dW\left( t\right) ,~ & in~\left(
0,T\right) \times \mathcal{O}, \\
sign\left( X\left( t\right) \right) \ln \left( | X\left( t\right)| +1\right)
=0, & on~\left( 0,T\right) \times \partial \mathcal{O}, \\
X\left( 0\right) =x, & in~\mathcal{O},%
\end{array}%
\right.  \label{equ}
\end{equation}%
where $\mathcal{O}$\ is an open bounded interval of $\mathbb{R}$, $x$ is an
initial datum and

\begin{equation*}
sign\left( x\right) =\left\{
\begin{array}{l}
\dfrac{x}{\left\vert x\right\vert },\quad \text{if \ }x\neq 0; \\
\left[ -1,1\right] ,\quad \text{if \ }x=0.%
\end{array}%
\right.
\end{equation*}

Here $W\left( t\right) $ is a cylindrical Wiener process on $L^{2}\left(
\mathcal{O}\right) $ of the form
\begin{equation*}
W\left( t\right) =\sum_{k=1}^{\infty }\beta _{k}\left( t\right) e_{k},\quad
t\geq 0
\end{equation*}%
for $\left\{ \beta _{k}\right\} $\ a sequence of independent standard
Brownian motion on\ a filtered probability space $\left( \Omega ,\mathcal{F}%
,\left\{ \mathcal{F}_{t}\right\} _{t\geq 0},\mathbb{P}\right) $ and $\left\{
e_{k}\right\} $ is a complete orthonormal system in $L^{2}\left( \mathcal{O}%
\right) $ of eigenfunctions of $-\Delta $ with Dirichlet homogeneous
boundary conditions. We denote by $\left\{ \lambda _{k}\right\} $ the
corresponding sequence of eigenvalues. The operator $Q\in L\left(
L^{2}\left( \mathcal{O}\right) \right) $ defined by
\begin{equation*}
\sqrt{Q}h=\sum_{k=1}^{\infty }\gamma _{k}\left\langle h,e_{k}\right\rangle
_{2}e_{k},\quad \forall ~h\in L^{2}\left( \mathcal{O}\right)
\end{equation*}%
where $\left\{ \gamma _{k}\right\} $\ is a sequence of positive numbers, is
symmetric, self-adjoint and nonnegative. Then the random forcing term is
\begin{equation*}
\sqrt{Q}dW\left( t\right) =\sum_{k=1}^{\infty }\gamma _{k}e_{k}d\beta
_{k}\left( t\right) ,\quad t\geq 0.
\end{equation*}

Because $\Psi :\mathbb{R} \rightarrow
%TCIMACRO{\U{211d} }%
%BeginExpansion
\mathbb{R}
%EndExpansion
,$ $\Psi \left( x\right) =sign\left( x \right) \ln \left(| x| +1\right) $ is
a maximal monotone operator, we can see that the operator defined by $%
A\left( x\right) =-\Delta \Psi \left( x\right) ,$ for all $x\in D\left(
A\right) ,$ where $D\left( A\right) =\left\{ x\in H^{-1}\left( \mathcal{O}%
\right) \cap L^{1}\left( \mathcal{O}\right) ;~\Psi \left( x\right) \in
H_{0}^{1}\left( \mathcal{O}\right) \right\} $ is maximal monotone in $%
H^{-1}\left( \mathcal{O}\right) \times H^{-1}\left( \mathcal{O}\right) $
(see \cite{brezis.operatori}).

In this paper we shall assume that the sequence $\left\{ \gamma _{k}\right\}
$ is such that

\begin{itemize}
\item[$\left( H_{1}\right) $] $\sum_{k=1}^{\infty }\gamma _{k}^{2}\lambda
_{k}^{2}<\infty $ and $\sum_{k=1}^{\infty }\gamma _{k}\lambda
_{k}^{3}<\infty $;
\end{itemize}

and

\begin{itemize}
\item[$\left( H_{2}\right) $] $\sqrt{Q}W\in C\left( \left[ 0,T\right] \times
\overline{\mathcal{O}}\right) $ $\mathbb{P}-a.s..$
\end{itemize}

Note that, for every $\omega $ fixed, we have, for some constant $C,$ that%
\begin{equation}
\underset{s\in \left[ 0,T\right] }{\sup }\left\vert \sqrt{Q}W\left( s\right)
\right\vert _{L^{\infty }\left( \mathcal{O}\right) }\leq C.  \label{spatiu}
\end{equation}

A similar result was proven in \cite{parab} for semilinear parabolic
stochastic equations and in \cite{strong} for porous media stochastic
equations.\bigskip

Denote by $g:\mathbb{R} \rightarrow \mathbb{R},\ $%
\begin{equation*}
g\left( x\right) =\left( |x|+1\right) \ln \left( |x|+1\right) -|x|,
\end{equation*}%
and note that, in our case, $\partial g=\Psi .$

\begin{definition}
\label{def}An adapted stochastic process
\begin{equation*}
X\in C_{W}\left( \left[ 0,T\right] ;H^{-1}\left( \mathcal{O}\right) \right)
\cap L^{p}\left( \left( 0,T\right) \times \mathcal{O\times }\Omega \right)
,\quad p\geq 4,
\end{equation*}%
is said to be a solution to equation (\ref{equ}) if%
\begin{eqnarray*}
&&\frac{1}{2}\left\vert X\left( t\right) -\sqrt{Q}W\left( t\right) -Z\left(
t\right) \right\vert _{-1}^{2}+\int_{0}^{t}\int_{\mathcal{O}}g\left( X\left(
s\right) \right) d\xi ds\medskip \\
&&\quad \quad \quad \quad +\int_{0}^{t}\int_{\mathcal{O}}\left( -\Delta
\right) ^{-1}\frac{\partial }{\partial s}Z\left( s\right) \left( X\left(
s\right) -\sqrt{Q}W\left( s\right) -Z\left( s\right) \right) d\xi ds\medskip
\\
&\leq &\frac{1}{2}\left\vert x-Z\left( 0\right) \right\vert
_{-1}^{2}+\int_{0}^{t}\int_{\mathcal{O}}g\left( Z\left( s\right) +\sqrt{Q}%
W\left( s\right) \right) d\xi ds,\quad \mathbb{P-}a.s.
\end{eqnarray*}%
for every starting point $x\in L^{p}\left( \mathcal{O}\right) $ and for all
adapted stochastic processes
\begin{equation*}
Z\in C_{W}\left( \left[ 0,T\right] ;H^{-1}\left( \mathcal{O}\right) \right)
,
\end{equation*}%
such that, for every $\omega \in \Omega $ fixed, satisfies

\begin{itemize}
\item[i)] $Z\in C\left( \left[ 0,T\right] ;L^{2}\left( \mathcal{O}\right)
\right) ;\medskip $

\item[ii)] $Z^{\prime }\in L^{2}\left( 0,T;H^{-1}\left( \mathcal{O}\right)
\right) ;\medskip $

\item[iii)] $g\left( Z+\sqrt{Q}W\right) \in L^{1}\left( \left( 0,T\right)
\times \mathcal{O}\right) .\medskip $
\end{itemize}
\end{definition}

Definition 1 resembles the classical definition of a mild (integral)
solution to deterministic variational inequality (see, e.g., \cite{analys}).
For stochastic differential equations a slightly different version was used
in \cite{BDPR} and \cite{AR}. We can easily see that a solution in the sense
of [\cite{concise}, Definition 4.2.1] is also a solution in the sense of
Definition \ref{def} above.

\textbf{Context}

Existence results for equation
\begin{equation*}
\left\{
\begin{array}{ll}
dX\left( t\right) -\Delta \Psi \left( X\left( t\right) \right) dt=\sqrt{Q}%
dW\left( t\right) ,~ & in~\left( 0,T\right) \times \mathcal{O}, \\
\Psi \left( X\left( t\right) \right) =0, & on~\left( 0,T\right) \times
\partial \mathcal{O}, \\
X\left( 0\right) =x, & in~\mathcal{O},%
\end{array}%
\right.
\end{equation*}%
were obtained in \cite{nonneg} for$\ \Psi $ monotonically increasing,
continuous, with $\Psi \left( 0\right) =0,$ and satisfying the following
growth conditions
\begin{equation*}
\Psi ^{\prime }\left( r\right) \leq \alpha _{1}\left\vert r\right\vert
^{m-1}+\alpha _{2}\text{ and }\int_{0}^{r}\Psi \left( s\right) ds\geq \alpha
_{3}\left\vert r\right\vert ^{m+1}+\alpha _{4},
\end{equation*}%
for all $r\in \mathbb{R}$, where $\alpha _{1},\alpha _{2},\alpha _{4}\geq
0,~\alpha _{3}>0$ and $m\geq 1.$ This result was generalized in \cite%
{criticality}. Note that our case is not covered by those hypotheses.

An other existence result was proved in \cite{fast} for the operator
\begin{equation*}
\Psi \left( r\right) =sign\left( r\right) \left\vert r\right\vert ^{\theta
-1}\left( \log \left( \left\vert r\right\vert +1\right) \right) ^{s},
\end{equation*}%
$r\in \mathbb{R}$ and $\theta \in \left( 1,\infty \right) ,$ $s\in \left[
1,\infty \right) $ (see \cite{fast} Example 3.5). \newline
In the present paper we are considering the critical case $\theta =1$ which
was not covered, by using a different approach and a different definition of
the solution.

\section{The main result}

The main result of this work is the following

\begin{theorem}
\label{theorem} For all $x\in L^{p}\left( \mathcal{O}\right) ,$ $p\geq 4,$
equation (\ref{equ}) has an unique solution in the sense of Definition \ref%
{def}.
\end{theorem}

In order to prove this result we need some estimates that will be used for
both existence and uniqueness.

\textbf{A priori Estimates}

Denote by $\Psi :\mathbb{R}\rightarrow 2^{\mathbb{R}},~\Psi \left( x\right)
=sign\left( x\right) \ln \left( |x|+1\right) .$ Since $\Psi $ is a maximal
monotone operator we can consider the following approximating equation%
\begin{equation}
\left\{
\begin{array}{ll}
dX_{\varepsilon }\left( t\right) -\Delta \overline{\Psi }_{\varepsilon
}\left( X_{\varepsilon }\left( t\right) \right) dt=\sqrt{Q}dW\left( t\right)
,~ & in~\left( 0,T\right) \times \mathcal{O}, \\
\overline{\Psi }_{\varepsilon }\left( X_{\varepsilon }\left( t\right)
\right) =0, & on~\left( 0,T\right) \times \partial \mathcal{O}, \\
X_{\varepsilon }\left( 0\right) =x, & in~\mathcal{O},%
\end{array}%
\right.  \label{approx1}
\end{equation}%
where $\overline{\Psi }_{\varepsilon }\left( x\right) =\Psi _{\varepsilon
}\left( x\right) +\varepsilon x,~$for all $x\in \mathbb{R},$ and $\Psi
_{\varepsilon }$ is the Yosida approximation of $\Psi ,$ i.e.,%
\begin{equation*}
\Psi _{\varepsilon }\left( x\right) =\frac{1}{\varepsilon }\left( 1-\left(
1+\varepsilon \Psi \right) ^{-1}\right) \left( x\right) =\Psi \left( \left(
1+\varepsilon \Psi \right) ^{-1}x\right)
\end{equation*}%
for all $\varepsilon >0.$ We can take this approximation since $\Psi $ is a
maximal monotone operator (see e.g., \cite{analys}, \cite{new}).\newline
For each $\varepsilon >0$ fixed, equation (\ref{approx1}) has an unique
solution in the sense of \cite{infinite} or [\cite{concise}, Definition
4.2.1] (see Example 4.1.11 from \cite{concise}). Note that solution $%
X_{\varepsilon }$ to the approximation equation (\ref{approx1}) is in our
case a path-wise continuous, $H^{-1}\left( \mathcal{O}\right) $ - valued, $%
\left( \mathcal{F}_{t}\right) $ - adapted stochastic process. Clearly, this
is also solution in the sense of Definition \ref{def}.

Setting
\begin{equation*}
Y_{\varepsilon }\left( t\right) =X_{\varepsilon }\left( t\right) -\sqrt{Q}%
W\left( t\right) ,
\end{equation*}%
we may rewrite (\ref{approx1}) as a random equation%
\begin{equation}
\left\{
\begin{array}{ll}
dY_{\varepsilon }\left( t\right) -\Delta \overline{\Psi }_{\varepsilon
}\left( Y_{\varepsilon }\left( t\right) +\sqrt{Q}W\left( t\right) \right)
dt=0, & in~\left( 0,T\right) \times \mathcal{O}, \\
\overline{\Psi }_{\varepsilon }\left( Y_{\varepsilon }\left( t\right) +\sqrt{%
Q}W\left( t\right) \right) =0, & on~\left( 0,T\right) \times \partial
\mathcal{O}, \\
Y_{\varepsilon }\left( 0\right) =x, & in~\mathcal{O},%
\end{array}%
\right.  \label{approx2}
\end{equation}%
For each $\omega \in \Omega $ fixed, by classical existence theory for
nonlinear equation we have that equation (\ref{approx2}) has a unique
solution $Y_{\varepsilon }\in C\left( \left[ 0,T\right] ;L^{2}\left(
\mathcal{O}\right) \right) ,$ with $Y_{\varepsilon }^{\prime }\in
L^{2}\left( 0,T;H^{-1}\left( \mathcal{O}\right) \right) $ (see \cite{lions}
for the general result and \cite{strong} for a similar case).

By the It\^{o} formula with the function $x\mapsto \left\vert x\right\vert
_{2}^{2^{{}}}$\ we get from (\ref{approx1}) that%
\begin{equation*}
\mathbb{E}\left\vert X_{\varepsilon }\left( t,x\right) \right\vert
_{2}^{2}\leq \left\vert x\right\vert _{2}^{2}+t\sum_{k=1}^{\infty }\lambda
_{k}^{2}\gamma _{k}^{2}
\end{equation*}%
and then, we get that, for each $\omega \in \Omega $ fixed, we have%
\begin{equation}
\left\vert X_{\varepsilon }\left( t\right) \right\vert _{L^{2}\left(
\mathcal{O}\right) }^{2}\leq C\left( \omega \right) ,  \label{ito1}
\end{equation}%
for all $t\in \left[ 0,T\right] ,$ with $C$ independent of $\varepsilon $.
By assumption \textbf{H}$_{2},$ we can assume the same estimate holds for $%
Y_{\varepsilon },$ i.e.,%
\begin{equation}
\left\vert Y_{\varepsilon }\left( t\right) \right\vert _{L^{2}\left(
\mathcal{O}\right) }^{2}=\left\vert X_{\varepsilon }\left( t\right) -\sqrt{Q}%
W\left( t\right) \right\vert _{L^{2}\left( \mathcal{O}\right) }^{2}\leq C,
\label{ito2}
\end{equation}%
for all $t\in \left[ 0,T\right] ,$ with $C$ independent of $\varepsilon .$

\begin{lemma}
\label{lemma}There exists a constant $C$ (independent of $\varepsilon $)
such that for all $\omega \in \Omega $ fixed, we have that%
\begin{equation*}
\int_{0}^{T}\int_{\mathcal{O}}\left\vert \Psi _{\varepsilon }\left(
Y_{\varepsilon }\left( s\right) +\sqrt{Q}W\left( s\right) \right)
\right\vert d\xi ds\leq C,
\end{equation*}%
for all $t\in \left[ 0,T\right] .$
\end{lemma}

\begin{proof}[Proof of Lemma \protect\ref{lemma}]

Recall that $g:\mathbb{R} \rightarrow \mathbb{R},$ is defined by%
\begin{equation*}
g\left( x\right) =\left( |x|+1\right) \ln \left( |x|+1\right) -|x|
\end{equation*}%
and $g_{\varepsilon }$ is the Moreau-Yosida approximation of $g$. Since $%
\partial g=\Psi $ we have by [\cite{analys}, Theorem 2.2] that $\nabla
g_{\varepsilon }=\Psi _{\varepsilon }.$

From the definition of the subdifferential we have, for all $0<\lambda <1$
fixed and for all $\theta ,$ such that $\left\vert \theta \right\vert <%
\dfrac{\lambda }{2},$ the following inequality%
\begin{eqnarray*}
&&\Psi _{\varepsilon }\left( Y_{\varepsilon }+\sqrt{Q}W\right) \left(
Y_{\varepsilon }+\sqrt{Q}W-\theta -\lambda \right) \\
&\geq &g_{\varepsilon }\left( Y_{\varepsilon }+\sqrt{Q}W\right)
-g_{\varepsilon }\left( \theta +\lambda \right) \\
&\geq &g\left( J_{\varepsilon }\left( Y_{\varepsilon }+\sqrt{Q}W\right)
\right) -g\left( \theta +\lambda \right) ,
\end{eqnarray*}%
$a.e.$ on $\left[ 0,T\right] \times \mathcal{O}$.

Taking into account that
\begin{equation*}
0\leq g\left( x\right) \leq x^{2},\text{ for all }x\in \mathbb{R},\
\end{equation*}%
we obtain that%
\begin{eqnarray*}
\Psi _{\varepsilon }\left( Y_{\varepsilon }+\sqrt{Q}W\right) \left(
Y_{\varepsilon }+\sqrt{Q}W-\theta -\lambda \right) &\geq &-\left( \theta
+\lambda \right) ^{2} \\
&>&-\frac{9\lambda ^{2}}{4}>-\frac{9}{4},
\end{eqnarray*}%
$a.e.$ on $\left[ 0,T\right] \times \mathcal{O}.$\newline
By taking
\begin{equation*}
\theta =\frac{\lambda }{2}\frac{\Psi _{\varepsilon }\left( Y_{\varepsilon }+%
\sqrt{Q}W\right) }{\left\vert \Psi _{\varepsilon }\left( Y_{\varepsilon }+%
\sqrt{Q}W\right) \right\vert }
\end{equation*}%
it follows that%
\begin{equation*}
\left\vert \Psi _{\varepsilon }\left( Y_{\varepsilon }+\sqrt{Q}W\right)
\right\vert \leq \frac{2}{\lambda }\Psi _{\varepsilon }\left( Y_{\varepsilon
}+\sqrt{Q}W\right) \left( Y_{\varepsilon }+\sqrt{Q}W-\lambda \right) +\frac{9%
}{2\lambda }
\end{equation*}%
$a.e.$ on $\left[ 0,T\right] \times \mathcal{O}.$

Consequently we obtain that%
\begin{eqnarray}
&&\int_{0}^{T}\int_{\mathcal{O}}\left\vert \Psi _{\varepsilon }\left(
Y_{\varepsilon }(s)+\sqrt{Q}W\right) (s)\right\vert d\xi ds\quad \quad \quad
\quad \quad \quad \quad \quad \quad \quad \quad \quad \quad \quad
\label{estimm} \\
&\leq &\frac{2}{\lambda }\int_{0}^{T}\int_{\mathcal{O}}\Psi _{\varepsilon
}\left( Y_{\varepsilon }+\sqrt{Q}W\right) \left( Y_{\varepsilon }+\sqrt{Q}%
W-\lambda \right) d\xi ds+\frac{9}{2\lambda }\left\vert \mathcal{O}%
\right\vert T  \notag \\
&\leq &\frac{2}{\lambda }\int_{0}^{T}\int_{\mathcal{O}}\left[ \left( -\Delta
\right) ^{-1}\left( -\frac{\partial }{\partial s}Y_{\varepsilon }\right)
-\varepsilon \left( Y_{\varepsilon }+\sqrt{Q}W\right) \right]  \notag \\
&&\quad \quad \quad \quad \quad \quad \quad \quad \times \left(
Y_{\varepsilon }+\sqrt{Q}W-\lambda \right) d\xi ds+\frac{9}{2\lambda }%
\left\vert \mathcal{O}\right\vert T  \notag \\
&\leq &\frac{2}{\lambda }\int_{0}^{T}\int_{\mathcal{O}}\left[ \left( -\Delta
\right) ^{-1}\left( -\frac{\partial }{\partial s}Y_{\varepsilon }\right) %
\right] \left( Y_{\varepsilon }+\sqrt{Q}W-\lambda \right) d\xi ds  \notag \\
&&\quad \quad \quad \quad \quad \quad -\frac{2\varepsilon }{\lambda }%
\int_{0}^{T}\int_{\mathcal{O}}\left( Y_{\varepsilon }+\sqrt{Q}W\right)
^{2}d\xi ds  \notag \\
&&\quad \quad \quad \quad \quad \quad +2\varepsilon \int_{0}^{T}\int_{%
\mathcal{O}}\left( Y_{\varepsilon }+\sqrt{Q}W\right) d\xi ds+\frac{9}{%
2\lambda }\left\vert \mathcal{O}\right\vert T  \notag \\
&\leq &\frac{2}{\lambda }\int_{0}^{T}\int_{\mathcal{O}}\left[ \left( -\Delta
\right) ^{-1}\left( -\frac{\partial }{\partial s}Y_{\varepsilon }\right) %
\right] \left( Y_{\varepsilon }+\sqrt{Q}W-\lambda \right) d\xi ds+C.  \notag
\end{eqnarray}%
Now it is sufficient to show boundedness for%
\begin{eqnarray}
&&\frac{2}{\lambda }\int_{0}^{T}\int_{\mathcal{O}}\left( -\Delta \right)
^{-1}\left( -\frac{\partial }{\partial s}Y_{\varepsilon }\right) \left(
Y_{\varepsilon }+\sqrt{Q}W-\lambda \right) d\xi ds  \label{estim} \\
&=&\frac{2}{\lambda }\int_{0}^{T}\int_{\mathcal{O}}\left( -\Delta \right)
^{-1}\left( -\frac{\partial }{\partial s}Y_{\varepsilon }\right)
Y_{\varepsilon }d\xi ds  \notag \\
&&+\frac{2}{\lambda }\int_{0}^{T}\int_{\mathcal{O}}\left( -\Delta \right)
^{-1}\left( -\frac{\partial }{\partial s}Y_{\varepsilon }\right) \sqrt{Q}%
Wd\xi ds  \notag \\
&&-2\int_{0}^{T}\int_{\mathcal{O}}\left( -\Delta \right) ^{-1}\left( -\frac{%
\partial }{\partial s}Y_{\varepsilon }\right) d\xi ds\overset{denote}{=}%
I_{1}+I_{2}+I_{3.}  \notag
\end{eqnarray}

Firstly, we have that%
\begin{eqnarray*}
I_{1} &=&\frac{2}{\lambda }\int_{0}^{T}\int_{\mathcal{O}}\left( -\Delta
\right) ^{-1}\left( -\frac{\partial }{\partial s}Y_{\varepsilon }\right)
Y_{\varepsilon }d\xi ds=\int_{0}^{T}\left\langle -\frac{\partial }{\partial s%
}Y_{\varepsilon },Y_{\varepsilon }\right\rangle _{-1}ds \\
&=&-\frac{1}{2}\left( \left\vert Y_{\varepsilon }\left( T\right) \right\vert
_{_{-1}}^{2}-\left\vert Y_{\varepsilon }\left( 0\right) \right\vert
_{_{-1}}^{2}\right) \leq C.
\end{eqnarray*}%
For the second term%
\begin{equation*}
I_{2}=\frac{2}{\lambda }\int_{0}^{T}\int_{\mathcal{O}}\left( -\Delta \right)
^{-1}\left( -\frac{\partial }{\partial s}Y_{\varepsilon }\right) \sqrt{Q}%
Wd\xi ds,
\end{equation*}%
we may choose, for $\alpha >0,$ small enough, a decomposition $%
0=t_{0}<t_{1}<...<t_{i}<t_{i+1}<...<t_{N}=T$ such that, for all $t,s\in %
\left[ t_{i},t_{i+1}\right] ,$ we have%
\begin{equation*}
\left\vert \sqrt{Q}W\left( t\right) -\sqrt{Q}W\left( s\right) \right\vert
_{L^{\infty }\left( \mathcal{O}\right) }<\alpha .
\end{equation*}%
Consequently, we may write%
\begin{eqnarray*}
I_{2} &=&\frac{2}{\lambda }\sum_{i=0}^{N-1}\int_{t_{i}}^{t_{i+1}}\int_{%
\mathcal{O}}\left( -\Delta \right) ^{-1}\left( -\frac{\partial }{\partial s}%
Y_{\varepsilon }\right) \sqrt{Q}Wd\xi ds \\
&=&\frac{2}{\lambda }\sum_{i=0}^{N-1}\int_{t_{i}}^{t_{i+1}}\int_{\mathcal{O}%
}\left( -\Delta \right) ^{-1}\left( -\frac{\partial }{\partial s}%
Y_{\varepsilon }\right) \left( \sqrt{Q}W\left( s\right) -\sqrt{Q}W\left(
t_{i}\right) \right) d\xi ds \\
&&\quad \quad +\frac{2}{\lambda }\sum_{i=0}^{N-1}\int_{t_{i}}^{t_{i+1}}\int_{%
\mathcal{O}}\left( -\Delta \right) ^{-1}\left( -\frac{\partial }{\partial s}%
Y_{\varepsilon }\right) \sqrt{Q}W\left( t_{i}\right) d\xi ds \\
&\leq &\frac{2\alpha }{\lambda }\int_{0}^{T}\left\vert \overline{\Psi }%
_{\varepsilon }\left( Y_{\varepsilon }\left( s\right) +\sqrt{Q}W\left(
s\right) \right) \right\vert _{L^{1}\left( \mathcal{O}\right) }ds \\
&&\quad \quad -\frac{2}{\lambda }\sum_{i=0}^{N-1}\int_{t_{i}}^{t_{i+1}}\left%
\langle \frac{\partial }{\partial s}Y_{\varepsilon }\left( s\right) ,\sqrt{Q}%
W\left( t_{i}\right) \right\rangle _{-1}ds \\
&\leq &\frac{2\alpha }{\lambda }\int_{0}^{T}\left\vert \Psi _{\varepsilon
}\left( Y_{\varepsilon }\left( s\right) +\sqrt{Q}W\left( s\right) \right)
\right\vert _{L^{1}\left( \mathcal{O}\right) }ds+C.
\end{eqnarray*}%
Finally we have%
\begin{eqnarray*}
I_{3} &=&2\int_{0}^{T}\int_{\mathcal{O}}\left( -\Delta \right) ^{-1}\left(
\frac{\partial }{\partial s}Y_{\varepsilon }\right) d\xi ds \\
&=&2\int_{\mathcal{O}}\left( \left( -\Delta \right) ^{-1}Y_{\varepsilon
}\left( T\right) -\left( -\Delta \right) ^{-1}Y_{\varepsilon }\left(
0\right) \right) d\xi \\
&\leq &2\left( \left\vert \left( -\Delta \right) ^{-1}Y_{\varepsilon }\left(
T\right) \right\vert _{L^{1}\left( \mathcal{O}\right) }+\left\vert \left(
-\Delta \right) ^{-1}Y_{\varepsilon }\left( 0\right) \right\vert
_{L^{1}\left( \mathcal{O}\right) }\right) \leq C,
\end{eqnarray*}%
since $H_{0}^{1}\left( \mathcal{O}\right) \subset L^{1}\left( \mathcal{O}%
\right) $ and $\left\vert \left( -\Delta \right) ^{-1}x\right\vert
_{H_{0}^{1}\left( \mathcal{O}\right) }=\left\vert x\right\vert _{-1},$ for
all $x\in H^{-1}\left( \mathcal{O}\right) .$

Going back to (\ref{estim}) we get that%
\begin{eqnarray*}
&&\frac{2}{\lambda }\int_{0}^{T}\int_{\mathcal{O}}\left( -\Delta \right)
^{-1}\left( -\frac{\partial }{\partial s}Y_{\varepsilon }\right) \left(
Y_{\varepsilon }+\sqrt{Q}W-\lambda \right) d\xi ds \\
&\leq &\frac{2\alpha }{\lambda }\int_{0}^{T}\left\vert \Psi _{\varepsilon
}\left( Y_{\varepsilon }\left( s\right) +\sqrt{Q}W\left( s\right) \right)
\right\vert _{L^{1}\left( \mathcal{O}\right) }ds+C
\end{eqnarray*}%
and then, for $\alpha $ small enough and $\lambda $ fixed, we get from (\ref%
{estimm}) that%
\begin{equation*}
\int_{0}^{T}\int_{\mathcal{O}}\left\vert \Psi _{\varepsilon }\left(
Y_{\varepsilon }\left( s\right) +\sqrt{Q}W\left( s\right) \right)
\right\vert d\xi ds\leq C.
\end{equation*}%
The proof of the lemma is now complete.
\end{proof}

\begin{proof}[Proof of the main result]

\textbf{Existence}

We have to prove existence of the limit for $\left\{ Y_{\varepsilon
}\right\} _{\varepsilon }$ as $\varepsilon \rightarrow 0$ and consequently
we'll get existence of the solution for equation (\ref{equ}) in the sense of
Definition \ref{def}.\newline
From Lemma \ref{lemma} we get, for all $\omega \in \Omega $ fixed, that%
\begin{equation}
\left\vert \frac{\partial }{\partial t}\left( -\Delta \right)
^{-1}Y_{\varepsilon }\right\vert _{L^{1}\left( 0,T;L^{1}\left( \mathcal{O}%
\right) \right) }\leq C.
\end{equation}%
Since for $\mathcal{O}\subset \mathbb{R}$ we have $L^{1}\left( \mathcal{O}%
\right) \subset H^{-1}\left( \mathcal{O}\right) $, this leads to
\begin{eqnarray}
V_{0}^{T}\left( \left( -\Delta \right) ^{-1}Y_{\varepsilon }\right) &\leq
&\left\vert \frac{\partial }{\partial t}\left( -\Delta \right)
^{-1}Y_{\varepsilon }\right\vert _{L^{1}\left( 0,T;H^{-1}\left( \mathcal{O}%
\right) \right) }  \label{helly1} \\
&\leq &\left\vert \frac{\partial }{\partial t}\left( -\Delta \right)
^{-1}Y_{\varepsilon }\right\vert _{L^{1}\left( 0,T;L^{1}\left( \mathcal{O}%
\right) \right) }\leq C,  \notag
\end{eqnarray}%
for $C$ independent of $\varepsilon .$ We denoted by
\begin{equation*}
V_{0}^{T}\left( f\right) =\sup \sum\limits_{i=1}^{n}\left\vert f\left(
t_{i}\right) -f\left( t_{i-1}\right) \right\vert _{-1}
\end{equation*}%
where the supremum is taken over all partitions $D=\left\{
0=t_{0}<t_{1}<...<t_{n}=T\right\} $ of $\left[ 0,T\right] .$

On the other hand, by classical deterministical arguments we have that, for
each $\omega \in \Omega $ fixed
\begin{equation*}
\underset{t\in \left[ 0,T\right] }{\sup }\left\vert Y_{\varepsilon }\left(
t\right) \right\vert _{-1}^{2}\leq C\left( \omega \right)
\end{equation*}%
which leads to%
\begin{equation}
\underset{t\in \left[ 0,T\right] }{\sup }\left\vert \left( -\Delta \right)
^{-1}Y_{\varepsilon }\left( t\right) \right\vert _{H_{0}^{1}\left( \mathcal{O%
}\right) }^{2}\leq C\left( \omega \right) .  \label{ito}
\end{equation}%
Then, since $\left\{ \left( -\Delta \right) ^{-1}Y_{\varepsilon }\left(
t\right) \right\} _{\varepsilon }$ is bounded in $H_{0}^{1}\left( \mathcal{O}%
\right) $ and $H_{0}^{1}\left( \mathcal{O}\right) \subset H^{-1}\left(
\mathcal{O}\right) $ compactly, we have that
\begin{equation}
\left\{ \left( -\Delta \right) ^{-1}Y_{\varepsilon }\left( t\right) \right\}
_{\varepsilon }\text{ is compact in }H^{-1}\left( \mathcal{O}\right) ,\text{
for all }t\in \left[ 0,T\right] .  \label{helly2}
\end{equation}%
From (\ref{helly1}) and\ (\ref{helly2}) it follows, via Helly-Foia\c{s}
theorem (see Theorem 3.5 and Remark 3.2 from \cite{BPr} or page 238 from
\cite{MN}), that, on a subsequence, we have
\begin{equation*}
\left( -\Delta \right) ^{-1}Y_{\varepsilon }\left( t\right) \rightarrow
G\left( t\right) \text{ strongly in }H^{-1}\left( \mathcal{O}\right) ,\text{
for all }t\in \left[ 0,T\right] .
\end{equation*}%
Because $H_{0}^{1}\left( \mathcal{O}\right) \subset L^{2}\left( \mathcal{O}%
\right) \subset H^{-1}\left( \mathcal{O}\right) $ we have that%
\begin{eqnarray*}
\left\vert \left( -\Delta \right) ^{-1}Y_{\varepsilon }\left( t\right)
-G\left( t\right) \right\vert _{L^{2}\left( \mathcal{O}\right) } &\leq
&\varepsilon \left\vert \left( -\Delta \right) ^{-1}Y_{\varepsilon }\left(
t\right) -G\left( t\right) \right\vert _{H_{0}^{1}\left( \mathcal{O}\right) }
\\
&&+C\left( \varepsilon \right) \left\vert \left( -\Delta \right)
^{-1}Y_{\varepsilon }\left( t\right) -G\left( t\right) \right\vert _{-1}
\end{eqnarray*}%
for all $t\in \left[ 0,T\right] $ (see \cite{lions}, p. 58)\ and therefore
\begin{equation*}
\left( -\Delta \right) ^{-1}Y_{\varepsilon }\left( t\right) \rightarrow
G\left( t\right) \text{ strongly in }L^{2}\left( \mathcal{O}\right) ,\text{
for all }t\in \left[ 0,T\right] .
\end{equation*}%
Using (\ref{ito2}) we obtain that%
\begin{equation}
\left( -\Delta \right) ^{-1}Y_{\varepsilon }\left( t\right) \rightarrow
\left( -\Delta \right) ^{-1}Y\left( t\right) \text{ strongly in }L^{2}\left(
\mathcal{O}\right) ,\text{ for all }t\in \left[ 0,T\right] .  \label{tare2}
\end{equation}%
On the other hand we have
\begin{eqnarray*}
&&\left\vert Y_{\varepsilon }\left( t\right) -Y_{\lambda }\left( t\right)
\right\vert _{-1}^{2}\medskip \\
&=&\left\langle \left( -\Delta \right) ^{-1}Y_{\varepsilon }\left( t\right)
-\left( -\Delta \right) ^{-1}Y_{\lambda }\left( t\right) ,Y_{\varepsilon
}\left( t\right) -Y_{\lambda }\left( t\right) \right\rangle _{L^{2}\left(
\mathcal{O}\right) }\medskip \\
&\leq &\left\vert \left( -\Delta \right) ^{-1}Y_{\varepsilon }\left(
t\right) -\left( -\Delta \right) ^{-1}Y_{\lambda }\left( t\right)
\right\vert _{L^{2}\left( \mathcal{O}\right) }\left\vert Y_{\varepsilon
}\left( t\right) -Y_{\lambda }\left( t\right) \right\vert _{L^{2}\left(
\mathcal{O}\right) }.\medskip
\end{eqnarray*}%
Using (\ref{ito2}) and (\ref{tare2}) we get that%
\begin{equation}
Y_{\varepsilon }\left( t\right) \rightarrow Y\left( t\right) \text{ strongly
in }H^{-1}\left( \mathcal{O}\right) ,\text{ for all }t\in \left[ 0,T\right]
\label{convergenta}
\end{equation}%
and, since $X_{\varepsilon }=Y_{\varepsilon }+\sqrt{Q}W,$ with $\sqrt{Q}W\in
C\left( \left[ 0,T\right] \times \overline{\mathcal{O}}\right) $, we have
\begin{equation}
X_{\varepsilon }\left( t\right) \rightarrow X\left( t\right) \text{ strongly
in }H^{-1}\left( \mathcal{O}\right) ,\text{ for all }t\in \left[ 0,T\right] .
\label{c1}
\end{equation}

On the other hand, from (\ref{ito1}) we have, for every $\omega \in \Omega $
fixed, that
\begin{eqnarray*}
\int_{0}^{t}\int_{\mathcal{O}}g\left( \left( 1+\varepsilon \Psi \right)
^{-1}X_{\varepsilon }\left( t\right) \right) d\xi &\leq &\int_{0}^{t}\int_{%
\mathcal{O}}\left\vert \left( 1+\varepsilon \Psi \right) ^{-1}X_{\varepsilon
}\left( t\right) \right\vert ^{2}d\xi \\
&\leq &\int_{0}^{t}\int_{\mathcal{O}}\left\vert X_{\varepsilon }\left(
t\right) \right\vert ^{2}d\xi \leq C\left( \omega \right) ,\text{ for all}%
~t\in \left[ 0,T\right] ,
\end{eqnarray*}%
where $\left( 1+\varepsilon \Psi \right) ^{-1}$ is the resolvent of $\Psi .$

Since%
\begin{equation*}
\underset{|x|\rightarrow \infty }{\lim }\frac{g\left( x\right) }{\left\vert
x\right\vert }=\infty
\end{equation*}%
we have that $\left\{ \left( 1+\varepsilon \Psi \right) ^{-1}\left(
X_{\varepsilon }\left( t\right) \right) \right\} _{\varepsilon }$ is bounded
and equi-integrable in $L^{1}\left( \mathcal{O\times }\left( 0,T\right)
\right) .$ Then, by the Dunford-Pettis theorem, we get that the sequence is
weakly compact in $L^{1}\left( \mathcal{O\times }\left( 0,T\right) \right) .$
Hence, along a subsequence, again denoting by $\varepsilon $, we obtain that
\begin{equation}
\left( 1+\varepsilon \Psi \right) ^{-1}\left( X_{\varepsilon }\left(
t\right) \right) \rightharpoonup X\left( t\right) ,\text{ \ weakly in }%
L^{1}\left( \mathcal{O\times }\left( 0,T\right) \right) ,  \label{c2}
\end{equation}%
as $\varepsilon \rightarrow 0.$

We know that $X_{\varepsilon }=Y_{\varepsilon }+\sqrt{Q}W$ is also a
solution to equation (\ref{approx1}) in the sense of our definition, $i.e.,$%
\begin{eqnarray}
&&\frac{1}{2}\left\vert X_{\varepsilon }\left( t\right) -\sqrt{Q}W\left(
t\right) -Z\left( t\right) \right\vert _{-1}^{2}+\int_{0}^{t}\int_{\mathcal{O%
}}g_{\varepsilon }\left( X_{\varepsilon }\left( s\right) \right) d\xi
ds\medskip  \label{def-aprox} \\
&&\quad \quad \quad \quad +\int_{0}^{t}\int_{\mathcal{O}}\left( -\Delta
\right) ^{-1}\frac{\partial }{\partial s}Z\left( s\right) \left(
X_{\varepsilon }\left( s\right) -\sqrt{Q}W\left( s\right) -Z\left( s\right)
\right) d\xi ds\medskip  \notag \\
&\leq &\frac{1}{2}\left\vert x-Z\left( 0\right) \right\vert _{H^{-1}\left(
\mathcal{O}\right) }^{2}+\int_{0}^{t}\int_{\mathcal{O}}g_{\varepsilon
}\left( Z\left( s\right) +\sqrt{Q}W\left( s\right) \right) d\xi ds,\quad
\mathbb{P-}a.s.\medskip  \notag
\end{eqnarray}%
for all$~t\in \left[ 0,T\right] .$

We intend to take the $liminf$ for $\varepsilon \rightarrow 0$ in (\ref%
{def-aprox}).

Convergence of the first term is a direct consequence of (\ref{c1}) and for
the last term we only need to use classical properties of the Moreau-Yosida
approximation i.e.,\ $g\left( \left( 1+\varepsilon \Psi \right) ^{-1}\left(
x\right) \right) \leq g_{\varepsilon }\left( x\right) \leq g\left( x\right) $
for all $x\in \mathbb{R}$.

We shall discuss now the second and the third term of the left hand side.%
\newline

Since $\varphi :L^{1}\left( \mathcal{O\times }\left( 0,T\right) \right)
\rightarrow \mathbb{R}$ defined by $\varphi \left( x\right)
=\int_{0}^{t}\int_{\mathcal{O}}g\left( x\left( \xi ,s\right) \right) d\xi ds$
is weakly $l.s.c.$ on $L^{1}\left( \mathcal{O\times }\left( 0,T\right)
\right) $ (see \cite{analys} Proposition 2.9 and Proposition 2.12 from
Chapter 2) we get from (\ref{c2}) that
\begin{equation*}
\underset{\varepsilon \rightarrow 0}{\lim \inf }\int_{0}^{t}\int_{\mathcal{O}%
}g_{\varepsilon }\left( X_{\varepsilon }\left( s\right) \right) d\xi ds\geq
\int_{0}^{t}\int_{\mathcal{O}}g\left( X\left( s\right) \right) d\xi ds,\text{
\ for all}~t\in \left[ 0,T\right] ,
\end{equation*}%
and then we can pass to the $liminf$ for $\varepsilon \rightarrow 0$ in the
second term.\newline
The third term of the left hand side can be written as%
\begin{equation*}
\int_{0}^{t}\left\langle \frac{\partial }{\partial s}Z\left( s\right)
,Y_{\varepsilon }\left( s\right) -Z\left( s\right) \right\rangle _{-1}ds.
\end{equation*}%
From (\ref{convergenta}) we have that%
\begin{equation*}
\left\langle \frac{\partial }{\partial s}Z,Y_{\varepsilon }-Z\right\rangle
_{-1}\rightarrow \left\langle \frac{\partial }{\partial s}Z,Y-Z\right\rangle
_{-1},\text{ }a.e.\text{ on }\left[ 0,T\right] .
\end{equation*}%
On the other hand we can easily see that%
\begin{eqnarray*}
\left\langle \frac{\partial }{\partial s}Z\left( s\right) ,Y_{\varepsilon
}\left( s\right) -Z\left( s\right) \right\rangle _{-1} &\leq &\left\vert
\frac{\partial }{\partial s}Z\left( s\right) \right\vert _{-1}\left\vert
Y_{\varepsilon }\left( s\right) -Z\left( s\right) \right\vert _{-1} \\
&\leq &\left\vert \frac{\partial }{\partial s}Z\left( s\right) \right\vert
_{-1}\underset{s\in \left[ 0,T\right] }{ess\sup }\left\vert Y_{\varepsilon
}\left( s\right) -Z\left( s\right) \right\vert _{-1} \\
&\leq &C\left\vert \frac{\partial }{\partial s}Z\left( s\right) \right\vert
_{-1}\in L^{2}\left( 0,T\right) ,\text{ }a.e.\text{ on }\left[ 0,T\right] .
\end{eqnarray*}%
Now, by using Lebesgue's dominated convergence theorem we obtain that%
\begin{equation*}
\underset{\varepsilon \rightarrow 0}{\lim }\int_{0}^{t}\left\langle \frac{%
\partial }{\partial s}Z\left( s\right) ,Y_{\varepsilon }\left( s\right)
-Z\left( s\right) \right\rangle _{-1}ds=\int_{0}^{t}\left\langle \frac{%
\partial }{\partial s}Z\left( s\right) ,Y\left( s\right) -Z\left( s\right)
\right\rangle _{-1}ds.
\end{equation*}%
At this point we can take the $\lim \inf $ \ for $\varepsilon \rightarrow 0$
in (\ref{def-aprox}) and get that
\begin{eqnarray*}
&&\frac{1}{2}\left\vert X\left( t\right) -\sqrt{Q}W\left( t\right) -Z\left(
t\right) \right\vert _{-1}^{2}+\int_{0}^{t}\int_{\mathcal{O}}g\left( X\left(
s\right) \right) d\xi ds\medskip \\
&&\quad \quad \quad \quad +\int_{0}^{t}\int_{\mathcal{O}}\left( -\Delta
\right) ^{-1}\frac{\partial }{\partial s}Z\left( s\right) \left( X\left(
s\right) -\sqrt{Q}W\left( s\right) -Z\left( s\right) \right) d\xi ds\medskip
\\
&\leq &\frac{1}{2}\left\vert x-Z\left( 0\right) \right\vert
_{-1}^{2}+\int_{0}^{t}\int_{\mathcal{O}}g\left( Z\left( s\right) +\sqrt{Q}%
W\left( s\right) \right) d\xi ds,\quad \mathbb{P-}a.s.
\end{eqnarray*}%
for all $t\in \left[ 0,T\right] .$

By applying the It\^{o} formula to equation (\ref{approx1}) with the
function $x\mapsto \left\vert x\right\vert _{-1}^{2}$ and from (\ref{c1}) we
get that
\begin{equation*}
X_{\varepsilon }\rightarrow X,\text{ weakly in }L_{W}^{2}\left( \left[ 0,T%
\right] ;L^{2}\left( \Omega ;H^{-1}\left( \mathcal{O}\right) \right) \right)
,\text{ as }\varepsilon \rightarrow 0.
\end{equation*}

Then, arguing as in \cite{concise}, we may replace $X$ by a $H^{-1}\left(
\mathcal{O}\right) -$ continuous version and follows that the solution is
also an $\left( \mathcal{F}_{t}\right) $- adapted stochastic process.

On the other hand we can easily see that $X\in L^{p}\left( \Omega \times
\mathcal{O}\times \left[ 0,T\right] \right) $ arguing as in Lemma 3.1 from
\cite{criticality} and that conclude the proof of the existence.

\textbf{Uniqueness}

Consider $\overline{X}$ an arbitrary solution to equation (\ref{equ}) in the
sense of Definition \ref{def}.

The main idea of the proof is to take $Z=\left( 1-\mu \Delta \right)
^{-1}Y_{\varepsilon }=J_{\mu }Y_{\varepsilon }$ in Definition \ref{def} ,
for $Y_{\varepsilon }$ the solution to equation (\ref{aprox3}) below\ and $%
J_{\mu }$ the resolvent of the Laplacian.

Consider, for each $\omega \in \Omega $ fixed, the following approximating
equation%
\begin{equation}
\left\{
\begin{array}{ll}
dY_{\varepsilon }\left( t\right) -\Delta \overline{\Psi }_{\varepsilon
}\left( Y_{\varepsilon }\left( t\right) +\sqrt{Q}W\left( t\right) \right)
dt=0, & in~\left( 0,T\right) \times \mathcal{O}, \\
\overline{\Psi }_{\varepsilon }\left( Y_{\varepsilon }\left( t\right) +\sqrt{%
Q}W\left( t\right) \right) =0, & on~\left( 0,T\right) \times \partial
\mathcal{O}, \\
Y_{\varepsilon }\left( 0\right) =x, & in~\mathcal{O}.%
\end{array}%
\right.  \label{aprox3}
\end{equation}%
where $\overline{\Psi }_{\varepsilon }\left( x\right) =\Psi _{\varepsilon
}\left( x\right) +\varepsilon x,~~$for all $x\in \mathbb{R} $
and $\Psi _{\varepsilon }$ is the Yosida approximation of $\Psi ,$ for every
$\varepsilon >0.$ By classical existence theory, equation (\ref{aprox3}) has
a unique solution
\begin{equation*}
Y_{\varepsilon }\in C\left( \left[ 0,T\right] ;L^{2}\left( \mathcal{O}%
\right) \right) \cap L^{2}\left( 0,T;H_{0}^{1}\left( \mathcal{O}\right)
\right) ,
\end{equation*}%
with $Y_{\varepsilon }^{\prime }\in L^{2}\left( 0,T;H^{-1}\left( \mathcal{O}%
\right) \right) .$

For $\mu >0$ fixed we consider the resolvent of the Laplacian $J_{\mu
}:L^{2}\left( \mathcal{O}\right) \rightarrow L^{2}\left( \mathcal{O}\right)
, $ $J_{\mu }\left( x\right) =\left( 1-\mu \Delta \right) ^{-1}\left(
x\right) ,$ for all $x\in L^{2}\left( \mathcal{O}\right) .$

Since $J_{\mu }$ is differentiable we may denote by $DJ_{\mu }$ the Gateaux
differential and we see that
\begin{equation}
\frac{d }{d t}J_{\mu }\left( Y_{\varepsilon }\left( t\right) \right)
=DJ_{\mu }\left( Y_{\varepsilon }\left( t\right) \right) \frac{\partial }{%
\partial t}Y_{\varepsilon }\left( t\right) =J_{\mu }\left( \frac{\partial }{%
\partial t}Y_{\varepsilon }\left( t\right) \right)  \label{comut}
\end{equation}%
for all $t\in \left[ 0,T\right] .$

Now we can prove that $Z=J_{\mu }Y_{\varepsilon }$ satisfies $i),~...,~iii)$
from Definition \ref{def}.

We can easily see that $J_{\mu }Y_{\varepsilon }\in C\left( \left[ 0,T\right]
;L^{2}\left( \mathcal{O}\right) \right) $ and then $i)$ is satisfied.

Concerning $ii)$ we know that, for every $\varepsilon >0$ fixed, we have%
\begin{equation*}
\int_{0}^{t}\left\vert J_{\mu }\frac{\partial }{\partial s}Y_{\varepsilon
}\left( s\right) \right\vert _{-1}^{2}ds\leq \int_{0}^{t}\left\vert \frac{%
\partial }{\partial s}Y_{\varepsilon }\left( s\right) \right\vert
_{-1}^{2}ds
\end{equation*}%
and%
\begin{equation*}
\frac{\partial }{\partial s}Y_{\varepsilon }\in L^{2}\left( 0,T;H^{-1}\left(
\mathcal{O}\right) \right) .
\end{equation*}%
Hence, by (\ref{comut}), we get that%
\begin{equation*}
\frac{\partial }{\partial t}Z=\frac{d }{d t}J_{\mu }\left( Y_{\varepsilon
}\right) \in L^{2}\left( 0,T;H^{-1}\left( \mathcal{O}\right) \right) .
\end{equation*}

Property $iii)$ is a direct consequence of $0\leq g\left( x\right) \leq
x^{2},$ for all $x\in \left( -1,\infty \right) .$ Indeed, we have for every $%
\varepsilon $ and $\mu $ fixed that
\begin{equation*}
\int_{0}^{t}\int_{\mathcal{O}}g\left( J_{\mu }Y_{\varepsilon }\left(
s\right) +\sqrt{Q}W\right) d\xi ds\leq \int_{0}^{t}\int_{\mathcal{O}%
}\left\vert J_{\mu }Y_{\varepsilon }\left( s\right) +\sqrt{Q}W\right\vert
^{2}d\xi ds\leq \infty
\end{equation*}%
because $J_{\mu }Y_{\varepsilon }\in C\left( \left[ 0,T\right] ;L^{2}\left(
\mathcal{O}\right) \right) $ and $\sqrt{Q}W\in C\left( \left[ 0,T\right]
\times \overline{\mathcal{O}}\right) .$

Since $J_{\mu }Y_{\varepsilon }$ satisfies i), ii) and iii) we can write
Definition \ref{def} for the solution $\bar{X}$ with $Z=J_{\mu
}Y_{\varepsilon },$ i.e.,
\begin{eqnarray}
&&\ \ \ \frac{1}{2}\left\vert \bar{X}\left( t\right) -\sqrt{Q}W\left(
t\right) -J_{\mu }Y_{\varepsilon }\left( t\right) \right\vert
_{-1}^{2}+\int_{0}^{t}\int_{\mathcal{O}}g\left( \bar{X}\left( s\right)
\right) d\xi ds  \label{def_unic} \\
&&\quad +\int_{0}^{t}\int_{\mathcal{O}}\left( -\Delta \right) ^{-1}\frac{d }{%
d s}J_{\mu }Y_{\varepsilon }\left( s\right) \left( \bar{X}\left( s\right) -%
\sqrt{Q}W\left( s\right) -J_{\mu }Y_{\varepsilon }\left( s\right) \right)
d\xi ds  \notag \\
&\leq &\frac{1}{2}\left\vert x-J_{\mu }Y_{\varepsilon }\left( s\right)
\left( 0\right) \right\vert _{-1}^{2}+\int_{0}^{t}\int_{\mathcal{O}}g\left(
J_{\mu }Y_{\varepsilon }\left( s\right) +\sqrt{Q}W\left( s\right) \right)
d\xi ds,\quad \mathbb{P-}a.s..  \notag
\end{eqnarray}%
Applying $J_{\mu }$ (which is linear) to (\ref{aprox3}) we get that%
\begin{equation*}
\frac{d }{d t}\left[ J_{\mu }Y_{\varepsilon }\left( t\right) \right] +J_{\mu
}\left( -\Delta \right) \left( \overline{\Psi }_{\varepsilon }\left(
Y_{\varepsilon }\left( t\right) +\sqrt{Q}W\left( t\right) \right) \right)
=0.
\end{equation*}%
By Proposition VII 2, $a_{1})$ and $a_{2})$ from \cite{brezis.analiza} we
can rewrite this equation as follows%
\begin{equation}
\frac{d }{d t}\left[ J_{\mu }Y_{\varepsilon }\left( t\right) \right] -\Delta
\left( J_{\mu }\overline{\Psi }_{\varepsilon }\left( Y_{\varepsilon }\left(
t\right) +\sqrt{Q}W\left( t\right) \right) \right) =0.  \label{ecc}
\end{equation}%
Using (\ref{ecc}) in the third term of the left-hand side of (\ref{def_unic}%
) we can rewrite it%
\begin{eqnarray*}
&&\int_{0}^{t}\int_{\mathcal{O}}\left( -\Delta \right) ^{-1}\frac{d }{d s}%
J_{\mu }Y_{\varepsilon }\left( s\right) \left( \bar{X}\left( s\right) -\sqrt{%
Q}W\left( s\right) -J_{\mu }Y_{\varepsilon }\left( s\right) \right) d\xi
ds\medskip \\
&=&\int_{0}^{t}\int_{\mathcal{O}}J_{\mu }\overline{\Psi }_{\varepsilon
}\left( Y_{\varepsilon }+\sqrt{Q}W\right) \left( J_{\mu }Y_{\varepsilon
}\left( s\right) -\bar{X}\left( s\right) +\sqrt{Q}W\left( s\right) \right)
d\xi ds\medskip \\
&=&\int_{0}^{t}\int_{\mathcal{O}}\overline{\Psi }_{\varepsilon }\left(
Y_{\varepsilon }+\sqrt{Q}W\right) \left( Y_{\varepsilon }+\sqrt{Q}W-J_{\mu
}\left( \overline{Y}+\sqrt{Q}W\right) \right) d\xi ds\medskip \\
&&+\int_{0}^{t}\int_{\mathcal{O}}\overline{\Psi }_{\varepsilon }\left(
Y_{\varepsilon }+\sqrt{Q}W\right) \left( J_{\mu }\sqrt{Q}W-\sqrt{Q}W\right)
d\xi ds\medskip \\
&&+\int_{0}^{t}\int_{\mathcal{O}}\overline{\Psi }_{\varepsilon }\left(
Y_{\varepsilon }+\sqrt{Q}W\right) \left( \left( 1-\mu \Delta \right)
^{-2}Y_{\varepsilon }-Y_{\varepsilon }\right) d\xi ds\medskip \\
&=&T_{1}+T_{2}+T_{3},
\end{eqnarray*}%
where $\overline{Y}=\bar{X}-\sqrt{Q}W.$

Since $\overline{\Psi }_{\varepsilon }\left( x\right) =\Psi _{\varepsilon
}\left( x\right) +\varepsilon x,$ where $\Psi _{\varepsilon }$ is the Yosida
approximation of $\Psi ,$ and $\overline{g}_{\varepsilon }\left( x\right)
=g_{\varepsilon }\left( x\right) +\varepsilon \frac{x^{2}}{2},$ where $%
g_{\varepsilon }$ is the Moreau-Yosida approximation of $g,$ we have that $%
\overline{g}_{\varepsilon }^{\prime }=\overline{\Psi }_{\varepsilon }$ and
then, by the definition of the subdifferential, we get that
\begin{equation*}
\overline{\Psi }_{\varepsilon }\left( x\right) \left( x-y\right) \geq
\overline{g}_{\varepsilon }\left( x\right) -\overline{g}_{\varepsilon
}\left( y\right) .
\end{equation*}

This leads to%
\begin{eqnarray*}
T_{1} &=&\int_{0}^{t}\int_{\mathcal{O}}\overline{\Psi }_{\varepsilon }\left(
Y_{\varepsilon }+\sqrt{Q}W\right) \left( Y_{\varepsilon }+\sqrt{Q}W-J_{\mu
}\left( \overline{Y}+\sqrt{Q}W\right) \right) d\xi ds \\
&\geq &\int_{0}^{t}\int_{\mathcal{O}}\left( \overline{g}_{\varepsilon
}\left( Y_{\varepsilon }+\sqrt{Q}W\right) -\overline{g}_{\varepsilon }\left(
J_{\mu }\left( \overline{Y}+\sqrt{Q}W\right) \right) \right) d\xi ds.
\end{eqnarray*}%
We also have
\begin{eqnarray*}
-T_{2} &=&\int_{0}^{t}\int_{\mathcal{O}}\overline{\Psi }_{\varepsilon
}\left( Y_{\varepsilon }+\sqrt{Q}W\right) \left( \sqrt{Q}W-J_{\mu }\sqrt{Q}%
W\right) d\xi ds \\
&\leq &\int_{0}^{t}\left\vert \overline{\Psi }_{\varepsilon }\left(
Y_{\varepsilon }+\sqrt{Q}W\right) \right\vert _{-1}\left\vert \sqrt{Q}%
W-J_{\mu }\sqrt{Q}W\right\vert _{H_{0}^{1}\left( \mathcal{O}\right) }ds \\
&\leq &\left\vert \overline{\Psi }_{\varepsilon }\left( Y_{\varepsilon }+%
\sqrt{Q}W\right) \right\vert _{L^{1}\left( \left( 0,T\right) \times \mathcal{%
O}\right) }\left\vert \sqrt{Q}W-J_{\mu }\sqrt{Q}W\right\vert _{L^{\infty
}\left( 0,T;H_{0}^{1}\left( \mathcal{O}\right) \right) } \\
&\leq &C\left\vert \sqrt{Q}W-J_{\mu }\sqrt{Q}W\right\vert _{L^{\infty
}\left( 0,T;H_{0}^{1}\left( \mathcal{O}\right) \right) },
\end{eqnarray*}%
using Lemma \ref{lemma} and the fact that $L^{1}\left( \mathcal{O}\right)
\subset H^{-1}\left( \mathcal{O}\right) ,$ for $\mathcal{O}\in \mathbb{R}$.%
\newline
From (\ref{aprox3}) follows that
\begin{eqnarray*}
T_{3} &=&\int_{0}^{t}\int_{\mathcal{O}}\left( -\Delta \right) ^{-1}\frac{%
\partial }{\partial s}Y_{\varepsilon }\left( s\right) \left( Y_{\varepsilon
}\left( s\right) -\left( 1-\mu \Delta \right) ^{-2}Y_{\varepsilon }\left(
s\right) \right) d\xi ds\medskip \\
&=&\int_{0}^{t}\left\langle \frac{\partial }{\partial s}Y_{\varepsilon
}\left( s\right) ,Y_{\varepsilon }\left( s\right) \right\rangle
_{-1}ds-\int_{0}^{t}\left\langle \frac{d }{d s}J_{\mu }Y_{\varepsilon
}\left( s\right) ,J_{\mu }Y_{\varepsilon }\left( s\right) \right\rangle
_{-1}ds\medskip \\
&=&\int_{0}^{t}\frac{\partial }{\partial s}\frac{1}{2}\left\vert
Y_{\varepsilon }\left( s\right) \right\vert _{-1}^{2}ds-\int_{0}^{t}\frac{%
\partial }{\partial s}\frac{1}{2}\left\vert J_{\mu }Y_{\varepsilon }\left(
s\right) \right\vert _{-1}^{2}ds\medskip \\
&=&\frac{1}{2}\left( \left\vert Y_{\varepsilon }\left( t\right) \right\vert
_{-1}^{2}-\left\vert J_{\mu }Y_{\varepsilon }\left( t\right) \right\vert
_{-1}^{2}\right) -\frac{1}{2}\left( \left\vert Y_{\varepsilon }\left(
0\right) \right\vert _{-1}^{2}-\left\vert J_{\mu }Y_{\varepsilon }\left(
0\right) \right\vert _{-1}^{2}\right) \medskip \\
&\geq &-\frac{1}{2}\left( \left\vert Y_{\varepsilon }\left( 0\right)
\right\vert _{-1}^{2}-\left\vert J_{\mu }Y_{\varepsilon }\left( 0\right)
\right\vert _{-1}^{2}\right) =-\frac{1}{2}\left( \left\vert x\right\vert
_{-1}^{2}-\left\vert J_{\mu }x\right\vert _{-1}^{2}\right)
\end{eqnarray*}%
since $\left\vert J_{\mu }Y_{\varepsilon }\left( t\right) \right\vert
_{-1}^{2}$ $\leq \left\vert Y_{\varepsilon }\left( t\right) \right\vert
_{-1}^{2}$ \ for all $t\in \left[ 0,T\right] $ and $x$ is the starting point
of the problem.\newline
Going back to (\ref{def_unic}) we get that
\begin{eqnarray}
&&\quad \frac{1}{2}\left\vert \overline{Y}\left( t\right) -J_{\mu
}Y_{\varepsilon }\left( t\right) \right\vert _{-1}^{2}+\int_{0}^{t}\int_{%
\mathcal{O}}g\left( \overline{X}\right) d\xi ds+\int_{0}^{t}\int_{\mathcal{O}%
}\overline{g}_{\varepsilon }\left( X_{\varepsilon }\right) d\xi ds
\label{dddd} \\
&\leq &\frac{1}{2}\left\vert x-J_{\mu }Y_{\varepsilon }\left( 0\right)
\right\vert _{-1}^{2}+\int_{0}^{t}\int_{\mathcal{O}}g\left( J_{\mu }%
\overline{X}\right) d\xi ds+\frac{\varepsilon }{2}\int_{0}^{t}\int_{\mathcal{%
O}}\left( J_{\mu }\overline{X}\right) ^{2}d\xi ds  \notag \\
&&\quad +\int_{0}^{t}\int_{\mathcal{O}}g\left( J_{\mu }Y_{\varepsilon }+%
\sqrt{Q}W\right) d\xi ds  \notag \\
&&\quad +C\left\vert \sqrt{Q}W-J_{\mu }\sqrt{Q}W\right\vert _{L^{\infty
}\left( 0,T;H_{0}^{1}\left( \mathcal{O}\right) \right) }  \notag \\
&&\quad +\frac{1}{2}\left( \left\vert x\right\vert _{-1}^{2}-\left\vert
J_{\mu }x\right\vert _{-1}^{2}\right) ,\quad \mathbb{P-}a.s..  \notag
\end{eqnarray}%
We firstly pass to the $liminf$ for $\varepsilon \rightarrow 0,$ with $\mu
>0 $ fixed, as follows.\newline
Arguing as we did in the proof of existence we have that%
\begin{equation*}
\underset{\varepsilon \rightarrow 0}{\lim \inf }\int_{0}^{t}\int_{\mathcal{O}%
}\overline{g}_{\varepsilon }\left( X_{\varepsilon }\left( s\right) \right)
d\xi ds\geq \int_{0}^{t}\int_{\mathcal{O}}g\left( X\left( s\right) \right)
d\xi ds,\text{ \ for all }t\in \left[ 0,T\right] .
\end{equation*}%
We shall now pass to the $liminf$ for $\varepsilon \rightarrow 0,$ with $\mu
>0 $ fixed, in
\begin{equation*}
\int_{0}^{t}\int_{\mathcal{O}}g\left( J_{\mu }Y_{\varepsilon }\left(
s\right) +\sqrt{Q}W\left( s\right) \right) d\xi ds.
\end{equation*}%
We know by (\ref{ito2}) that $\left\{ Y_{\varepsilon }\right\} _{\varepsilon
}$ is bounded in $C\left( \left[ 0,T\right] ;L^{2}\left( \mathcal{O}\right)
\right) $ and considering (\ref{convergenta}) we get that
\begin{equation*}
Y_{\varepsilon }\left( t\right) \rightarrow Y\left( t\right) ,\text{ weakly
in }L^{2}\left( \mathcal{O}\right) ,\text{\ \ for all }t\in \left[ 0,T\right]
,
\end{equation*}%
as $\varepsilon \rightarrow 0.$ On the other hand, we know that, for every $%
\mu >0$ fixed, we have that $J_{\mu }:L^{2}\left( \mathcal{O}\right)
\rightarrow L^{2}\left( \mathcal{O}\right) $ is compact and then
\begin{equation*}
J_{\mu }Y_{\varepsilon }\left( t\right) \rightarrow J_{\mu }Y\left( t\right)
,\text{ strongly in }L^{2}\left( \mathcal{O}\right) ,\text{\ \ for all }t\in %
\left[ 0,T\right] ,
\end{equation*}%
as $\varepsilon \rightarrow 0.$\newline
Since $g:\mathbb{R} \rightarrow \mathbb{R}$ is continuous and $\sqrt{Q}W\in
C\left( \left[ 0,T\right] \times \overline{\mathcal{O}}\right) $ we have that%
\begin{equation}
g\left( J_{\mu }Y_{\varepsilon }+\sqrt{Q}W\right) \rightarrow g\left( J_{\mu
}Y+\sqrt{Q}W\right) ,~a.e.\text{ on }\left[ 0,T\right] \times \mathcal{O}.
\label{lebesgue1}
\end{equation}%
We also know that $J_{\mu }:L^{2}\left( \mathcal{O}\right) \rightarrow
C\left( \overline{\mathcal{O}}\right) $ is continuous and then $\left\{
J_{\mu }Y_{\varepsilon }+\sqrt{Q}W\right\} _{\varepsilon }$ is bounded in $%
C\left( \left[ 0,T\right] \times \overline{\mathcal{O}}\right) .$ We obtain
that
\begin{equation}
g\left( J_{\mu }Y_{\varepsilon }+\sqrt{Q}W\right) \leq \left\vert J_{\mu
}Y_{\varepsilon }+\sqrt{Q}W\right\vert ^{2}\leq C,~a.e.\text{ on }\left[ 0,T%
\right] \times \mathcal{O}.  \label{lebesgue2}
\end{equation}%
Consequently, by Lebesgue's dominated convergence theorem, we get from (\ref%
{lebesgue1}) and (\ref{lebesgue2}) that%
\begin{equation*}
\underset{\varepsilon \rightarrow 0}{\lim }\int_{0}^{t}\int_{\mathcal{O}%
}g\left( J_{\mu }Y_{\varepsilon }+\sqrt{Q}W\right) d\xi ds=\int_{0}^{t}\int_{%
\mathcal{O}}g\left( J_{\mu }Y+\sqrt{Q}W\right) d\xi ds.
\end{equation*}%
Going back to (\ref{dddd}) and passing to the $liminf$ for $\varepsilon
\rightarrow 0,$ with $\mu >0$ fixed, we get that
\begin{eqnarray}
&&\quad \frac{1}{2}\left\vert \overline{Y}\left( t\right) -J_{\mu }Y\left(
t\right) \right\vert _{-1}^{2}\medskip +\int_{0}^{t}\int_{\mathcal{O}%
}g\left( X\right) d\xi ds+\int_{0}^{t}\int_{\mathcal{O}}g\left( \overline{X}%
\right) d\xi ds  \label{defff} \\
&\leq &\frac{1}{2}\left\vert x-J_{\mu }x\right\vert
_{-1}^{2}+\int_{0}^{t}\int_{\mathcal{O}}g\left( J_{\mu }\overline{X}\right)
d\xi ds+\int_{0}^{t}\int_{\mathcal{O}}g\left( J_{\mu }Y+\sqrt{Q}W\right)
d\xi ds\medskip  \notag \\
&&\quad +C\left\vert \sqrt{Q}W-J_{\mu }\sqrt{Q}W\right\vert _{L^{\infty
}\left( 0,T;H_{0}^{1}\left( \mathcal{O}\right) \right) }  \notag \\
&&\quad +\frac{1}{2}\left( \left\vert x\right\vert _{-1}^{2}-\left\vert
J_{\mu }x\right\vert _{-1}^{2}\right) ,\quad \mathbb{P-}a.s..  \notag
\end{eqnarray}%
We shall now pass to the $liminf$ for $\mu \rightarrow 0$.

Firstly we discuss the following term
\begin{equation*}
\int_{0}^{t}\int_{\mathcal{O}}\left[ g\left( J_{\mu }\overline{X}\left(
s\right) \right) -g\left( \overline{X}\left( s\right) \right) \right] d\xi
ds.
\end{equation*}%
Since
\begin{equation*}
J_{\mu }\overline{X}\left( s\right) \rightarrow \overline{X}\left( s\right) ,%
\text{ strongly in }L^{2}\left( \mathcal{O}\right) ,\ \text{\ for all }s\in %
\left[ 0,T\right]
\end{equation*}%
as $\mu \rightarrow 0$ and since $g:\mathbb{R} \rightarrow \mathbb{R}$ is
continuous, we have that%
\begin{equation*}
g\left( J_{\mu }\overline{X}\left( s\right) \right) \rightarrow g\left(
\overline{X}\left( s\right) \right) ,\text{ }a.e.\text{ on }\mathcal{O},\
\text{\ for all }s\in \left[ 0,T\right]
\end{equation*}%
as $\mu \rightarrow 0.$\newline
By the Egorov theorem we get that, for every $\delta >0,$ there exists a set
$E_{\delta }$ with the Lebesgue's measure $\left\vert E_{\delta }\right\vert
<\delta ,$ such that
\begin{equation*}
g\left( J_{\mu }\overline{X}\right) -g\left( \overline{X}\right) \rightarrow
0,\text{ uniformly on }\mathcal{O}\backslash E_{\delta }
\end{equation*}%
as $\mu \rightarrow 0,$ for all $t\in \left[ 0,T\right] $ and then
\begin{equation*}
\underset{\mu \rightarrow 0}{\lim }\int_{\mathcal{O}}\left[ g\left( J_{\mu }%
\overline{X}\right) -g\left( \overline{X}\right) \right] d\xi =\underset{\mu
\rightarrow 0}{\lim }\int_{E_{\delta }}\left[ g\left( J_{\mu }\overline{X}%
\right) -g\left( \overline{X}\right) \right] d\xi
\end{equation*}%
for all $t\in \left[ 0,T\right] .$

On the other hand we have%
\begin{eqnarray*}
&&\int_{E_{\delta }}\left[ g\left( J_{\mu }\overline{X}\right) -g\left(
\overline{X}\right) \right] d\xi \\
&\leq &\left( \int_{E_{\delta }}1d\xi \right) ^{1/2}\left( \int_{E_{\delta
}} \left[ g\left( J_{\mu }\overline{X}\right) -g\left( \overline{X}\right) %
\right] ^{2}d\xi \right) ^{1/2} \\
&\leq &2\left\vert E_{\delta }\right\vert ^{1/2}\left( \int_{E_{\delta }}%
\overline{X}^{4}d\xi \right) ^{1/2}\leq C\delta ,~a.e.\text{ on }\left[ 0,T%
\right] .
\end{eqnarray*}%
(Indeed, we can easily see that $\left\vert J_{\mu }\overline{X}\right\vert
_{L^{4}\left( E_{\delta }\right) }\leq \left\vert \overline{X}\right\vert
_{L^{4}\left( E_{\delta }\right) }$ by using the same argument as in \cite%
{nonneg} to obtain $3.25$).\newline
Then, we get that
\begin{equation*}
\underset{\mu \rightarrow 0}{\lim }\int_{\mathcal{O}}\left[ g\left( J_{\mu }%
\overline{X}\right) -g\left( \overline{X}\right) \right] d\xi =0,~a.e.\text{
on }\left[ 0,T\right] .
\end{equation*}%
We also know from (\ref{ito1}) that
\begin{eqnarray*}
\int_{\mathcal{O}}\left[ g\left( J_{\mu }\overline{X}\left( t\right) \right)
-g\left( \overline{X}\left( t\right) \right) \right] d\xi &\leq &2\int_{%
\mathcal{O}}\left\vert \overline{X}\left( t\right) \right\vert ^{2}d\xi \\
&\leq &2\underset{t\in \left[ 0,T\right] }{\sup }\int_{\mathcal{O}%
}\left\vert \overline{X}\left( t\right) \right\vert ^{2}d\xi \leq C,~a.e.%
\text{ on }\left[ 0,T\right] .
\end{eqnarray*}%
Using Lebesgue's dominated convergence theorem we get that
\begin{equation*}
\underset{\mu \rightarrow 0}{\lim }\int_{0}^{t}\int_{\mathcal{O}}\left[
g\left( J_{\mu }\overline{X}\left( s\right) \right) -g\left( \overline{X}%
\left( s\right) \right) \right] d\xi ds=0.
\end{equation*}

Using the same argument we get that
\begin{equation*}
\underset{\mu \rightarrow 0}{\lim }\int_{0}^{t}\int_{\mathcal{O}}\left[
g\left( J_{\mu }Y\left( s\right) +\sqrt{Q}W\left( s\right) \right) -g\left(
X\left( s\right) \right) \right] d\xi ds=0.
\end{equation*}

In order to conclude the proof we only need to mention that, for each $%
\omega \in \Omega $\ fixed, we have%
\begin{equation}
\underset{\mu \rightarrow 0}{\lim }\left\vert \sqrt{Q}W-J_{\mu }\sqrt{Q}%
W\right\vert _{L^{\infty }\left( 0,T;H_{0}^{1}\left( \mathcal{O}\right)
\right) }=0,  \label{cccc}
\end{equation}%
which is a consequence of the fact that $\sqrt{Q}W \in L^{\infty }\left(
0,T;H_{0}^{1}\left( \mathcal{O}\right) \right)$

Going back to (\ref{defff})\ we can pass to the $liminf$ for $\mu
\rightarrow 0 $ and get that, for each $\omega \in \Omega $ fixed, we have
\begin{equation*}
\left\vert \overline{X}\left( t\right) -X\left( t\right) \right\vert
_{-1}=0,
\end{equation*}%
for all $t\in \left[ 0,T\right] ,$ and that assure the uniqueness of the
solution.
\end{proof}

\begin{acknowledgement}
The author thanks the referee for the constructive comments and
suggestions.\bigskip
\end{acknowledgement}

\end{document}